\documentclass[12pt]{amsart}
\usepackage{epsfig,amssymb,latexsym}

\title[Modifications of Tutte-Grothendieck invariants]{Modifications of Tutte-Grothendieck invariants and Tutte polynomials}
\author{Martin Kochol}
\address{M\'U SAV, \v Stef\'anikova 49, 814 73 Bratislava 1, Slovakia}
\email{martin.kochol@mat.savba.sk}


\def\tg{\Phi}
\def\lsx{^{\rm ls}_{\xi}}
\def\isx{^{\rm is}_{\xi}}

\newtheorem{thm}{Theorem}
\newtheorem{lem}{Lemma}

\begin{document}
\maketitle
\begin{abstract}
We transform Tutte-Grothedieck invariants thus also Tutte polynomials on matroids so that the contraction-deletion rule for loops
(isthmuses) coincides with the general case.
\end{abstract}

\section{Introduction}

A {\it Tutte-Grothendieck invariant} (shortly a {\it T-G invariant}) $\tg$ is a mapping from the class of finite matroids to a commutative ring
$(R,+,\cdot,0,1)$ such that $\tg(M)=\tg(M')$ if $M$ is isomorphic to $M'$ and there are constants $\alpha_1,\beta_1,\alpha_2,\beta_2\in R$ such that
\begin{eqnarray}\label{e-tg1}
\begin{array}{lll}
&\tg(M) = 1  & \text{ if the ground set of $M$ is empty,}\\ 
&\tg(M) = \alpha_1 \cdot \tg(M-e) & \text{ if $e$ is an isthmus of $M$,} \\
&\tg(M) = \beta_1 \cdot \tg(M-e) & \text{ if $e$ is a loop of $M$,} \\
&\tg(M) = \alpha_2 \cdot \tg(M/e)  + \beta_2 \cdot \tg(M-e)  & \text{ otherwise,}
\end{array}
\end{eqnarray}
for every matroid $M$ and every element $e$ of $M$. We also say that $\tg$ is {\it determined} by the 4-tuple
$(\alpha_1,\beta_1,\alpha_2,\beta_2)$. In certain sense (see \cite{OW,BO}), all T-G invariants can be reduced from the {\it Tutte
polynomial} of $M$
\begin{eqnarray}\label{e-def1}
T(M;x,y) =\sum_{A\subseteq E}(x-1)^{r(M)-r(A)}(y-1)^{|A|-r(A)},
\end{eqnarray}
where $E$ and $r$ denote the ground set and rank function of $M$, respectively. This is very important invariant that encodes many properties of graphs and has applications in combinatorics, knot theory, statistical physics and coding theory (see cf. \cite{BEPS,BO,W}).

$M-e=M/e$ if $e$ is a loop or an isthmus of $M$. Thus the second (third) row of (\ref{e-tg1}) is contained in the fourth row if
$\alpha_1=\alpha_2+\beta_2$ ($\beta_1=\alpha_2+\beta_2$). In this case $\tg$ is called an {\it isthmus-smooth} ({\it loop-smooth})
T-G invariant.

We show that any T-G invariant can be transformed to an isthmus- and loop-smooth T-G invariants. The transformations are studied in
framework of matroid duality. Furthermore, we discuss modifications of covariance and convolution formulas known for the Tutte polynomial.
Notice that transformations into isthmus-smooth invariants are used by decomposition algorithms of T-G invariants in \cite{K}.

\section{General modifications}

\begin{lem}\label{tut-p2} Let $\alpha_1,\beta_1,\alpha_2,\beta_2$ be arbitrary elements of a commutative ring $(R,+,\cdot,0,1)$. Then $\tilde
T(M;\alpha_1,\beta_1,\alpha_2,\beta_2)=\alpha_2^{r(M)}\beta_2^{r^*(M)}T(M;\alpha_1/\alpha_2,\beta_1/\beta_2)$ is the unique T-G invariant
determined by $(\alpha_1,\beta_1,\alpha_2,\beta_2)$.
\end{lem}

\begin{proof} For any matroid $M$, denote $\tg(M)=\tilde T(M;\alpha_1,\beta_1,\alpha_2,\beta_2)$ (interpreting the formula as the
substitution $x_2=\alpha_2$, $y_2=\beta_2$ in the polynomial $\tilde T(M;x_1,y_1,x_2,y_2)$). We use that $(T;x,y)$ is determined by
$(x,y,1,1)$ and induction on $|E|$. The statement of lemma holds true if $|E|=0$, otherwise choose $e\in E$. If $e$ is an isthmus of $M$,
then
\begin{eqnarray*}\begin{array}{c}
\tg(M)=\alpha_2^{r(M)}\beta_2^{r^*(M)}T(M;\alpha_1/\alpha_2,\beta_1/\beta_2)=\\
\alpha_2^{r(M-e)+1}\beta_2^{r^*(M-e)}\alpha_1/\alpha_2T(M-e;\alpha_1/\alpha_2,\beta_1/\beta_2)=\alpha_1\tg(M-e)
\end{array}\end{eqnarray*}
by induction hypothesis. If $e$ is a loop of $M$, then 
\begin{eqnarray*}\begin{array}{c}
\tg(M)=\alpha_2^{r(M-e)}\beta_2^{r^*(M-e)+1}\beta_1/\beta_2T(M-e;\alpha_1/\alpha_2,\beta_1/\beta_2)= \beta_1\tg(M-e).
\end{array}\end{eqnarray*}
If $e$ is neither a loop nor an isthmus of $M$, then 
\begin{eqnarray*}\begin{array}{c}
\tg(M)=\alpha_2^{r(M/e)+1}\beta_2^{r^*(M/e)}T(M/e;\alpha_1/\alpha_2,\beta_1/\beta_2)+ \\
\alpha_2^{r(M-e)}\beta_2^{r^*(M-e)+1}T(M-e;\alpha_1/\alpha_2,\beta_1/\beta_2)= \alpha_2\tg(M/e)+\beta_2\tg(M-e).
\end{array}\end{eqnarray*}
This proves the statement.
\end{proof}

Lemma~\ref{tut-p2} also follows from results of Oxley and Welsh \cite{OW} (see \cite[Corollary 6.2.6]{BO}).

\begin{thm}\label{mod-t1} Let $\tg$ be a T-G invariant determined by $(\alpha_1,\beta_1,\alpha_2,\beta_2)$, $\beta_2\neq 0$, and $\xi\in R$ be a
multiple of $\beta_2$. Then $\tg\isx(M)=\xi^{|E|}\left(\frac{\alpha_1-\alpha_2}{\beta_2}\right )^{r^*(M)}\tg(M)$ is an isthmus-smooth
T-G invariant such that for every matroid $M$,
\begin{eqnarray*}
\begin{array}{ll}
\tg\isx(M) = 1  & \text{ if } E=\emptyset ,\\
\tg\isx(M) = \xi\beta_1 (\alpha_1-\alpha_2)/\beta_2 \tg\isx (M-e) & \text{ if $e$ is a loop of $M$,} \\
\tg\isx(M) =\xi\alpha_2 \tg\isx(M/e) + \xi(\alpha_1{-}\alpha_2)\tg\isx(M{-}e) & \text{ otherwise}.
\end{array}
\end{eqnarray*}
\end{thm}

\begin{proof} By Lemma~\ref{tut-p2}, $\tg(M)=\alpha_2^{r(M)}\beta_2^{r^*(M)}T(M;\alpha_1/\alpha_2,\beta_1/\beta_2)$ for each matroid $M$.
Setting $\zeta=\xi\left(\frac{\alpha_1-\alpha_2}{\beta_2}\right)$ and using equality $|E|=r(M)+r^*(M)$, we get
\begin{eqnarray*}
\begin{array}{c}
\tg\isx(M)=\xi^{r(M)}\zeta^{r^*(M)}\tg(M)=
(\xi\alpha_2)^{r(M)}(\zeta\beta_2)^{r^*(M)}T(M;\frac{\xi\alpha_1}{\xi\alpha_2},\frac{\zeta\beta_1}{\zeta\beta_2}),
\end{array}
\end{eqnarray*}
whence by Lemma~\ref{tut-p2}, $\tg\isx$ is a T-G invariant determined by $(\xi\alpha_1,\zeta\beta_1,\xi\alpha_2,\zeta\beta_2)$.
Furthermore, $\xi\alpha_1=\xi\alpha_2+\zeta\beta_2$, i.e., $\tg\isx$ is an isthmus-smooth T-G invariant.
\end{proof}

$\tg\isx$ is called the {\it $\xi$-isthmus-smooth modification} of $\tg$. Notice that if $\tg$ is an isthmus-smooth invariant
(i.e., if $\alpha_1=\alpha_2+\beta_2$), then $\tg\isx(M) =\xi^{|E|}\tg(M)$ for every matroid $M$.

If $R$ has zero divisors, then $\xi/\beta_2$ does not need to be unique. In this case we should formally replace fraction $\xi/\beta_2$ by
$\xi'$ where $\xi=\xi'\beta_2$. On the other hand if $\alpha_1-\alpha_2=\xi''\beta_2$, it suffices to replace fraction
$(\alpha_1-\alpha_2)/\beta_2$ by $\xi''$ and allow $\xi$ to be any element of $R$. If $R$ contains no zero divisors, we can extend $R$ into
its quotient field and allow $\xi$ to be any element of $R$, or any element of the quotient field.

If $\tg$ is a T-G invariant determined by $(\alpha_1,\beta_1,\alpha_2,\beta_2)$, then define $\tg^*$ as the T-G invariant determined by
$(\beta_1,\alpha_1,\beta_2,\alpha_2)$. Clearly, $\tg=(\tg^*)^*$. By Lemma~\ref{tut-p2},
$\tg(M)=\alpha_2^{r(M)}\beta_2^{r^*(M)}T(M;\alpha_1/\alpha_2,\beta_1/\beta_2)$ and
$\tg^*(M^*)=\beta_2^{r^*(M)}\alpha_2^{r(M)}T(M^*;\beta_1/\beta_2,\alpha_1/\alpha_2)$ for each matroid $M$. The covariance formula (see
\cite{BO}) is that $T(M;x,y) = T(M^*;y,x)$, whence
\begin{eqnarray}\label{e-tg4}
\tg(M) =\tg^*(M^*).
\end{eqnarray}

\begin{thm}\label{mod-t2} Let $\tg$ be a T-G invariant determined by $(\alpha_1,\beta_1,\alpha_2,\beta_2),\alpha_2\neq 0$, and $\xi\in R$ be a
multiple of $\alpha_2$. Then $\tg\lsx(M)=\xi^{|E|}\left(\frac{\beta_1-\beta_2}{\alpha_2}\right )^{r(M)}\tg(M)$ is a loop-smooth T-G
invariant such that for every matroid $M$,
\begin{eqnarray*}
\begin{array}{ll}
\tg\lsx(M) = 1  & \text{ if } E=\emptyset ,\\
\tg\lsx(M) = \xi\alpha_1 
(\beta_1-\beta_2)/\alpha_2 \tg\lsx (M-e) & \text{ if $e$ is an isthmus of $M$,} \\
\tg\lsx(M) =\xi(\beta_1{-}\beta_2)\tg\lsx(M/e) + \beta_2\tg\lsx(M{-}e) \hskip -2mm & \text{ otherwise}.
\end{array}
\end{eqnarray*}
\end{thm}

\begin{proof} Set $\tg\lsx =((\tg^*)\isx)^*$. By (\ref{e-tg4}) and  Theorem~\ref{mod-t1}, $\tg\lsx(M)=((\tg^*)\isx)^*(M)=(\tg^*)\isx(M^*)$.  Applying
Theorem~\ref{mod-t1} for $\tg^*$ and $M^*$, we get $(\tg^*)\isx(M^*)=\xi^{|E|} \left(\frac{\beta_1-\beta_2}{\alpha_2}\right
)^{r(M)}\tg^*(M^*)=\xi^{|E|}\left(\frac{\beta_1-\beta_2}{\alpha_2}\right )^{r(M)}\tg(M)$. Furthermore by definition of $\tg^*$ and
Theorem~\ref{mod-t1}, $((\tg^*)\isx)^*$ is determined by
$\left(\xi\alpha_1(\beta_1-\beta_2)/\alpha_2,\xi\beta_1,\xi(\beta_1-\beta_2),\xi\beta_2\right)$.
\end{proof}

Notice that $\tg\lsx=((\tg^*)\isx)^*$, whence $\tg\isx =((((\tg^*)^*)\isx)^*)^*=((\tg^*)\lsx)^*$. Thus 
\begin{eqnarray}\label{e-tg6}
\tg\lsx =((\tg^*)\isx)^* \text{ and } \tg\isx =((\tg^*)\lsx)^*.
\end{eqnarray}

$\tg\lsx$ is called the {\it $\xi$-loop-smooth modification} of $\tg$. If $\tg$ is an isthmus invariant, then
$\tg\lsx(M)=\xi^{|E|}\tg(M)$ for every matroid $M$.

In Theorems~\ref{mod-t1} and \ref{mod-t2} we have assumed that $\beta_2\neq 0$ and $\alpha_2\neq 0$, respectively. Let $l_M$ ($i_M$) denote
the number of loops (isthmuses) in a matroid $M$. If $\alpha_2=0$, then by (\ref{e-tg1}),
$\tg(M)=\alpha_1^{r(M)}\beta_1^{l_M}\beta_2^{r^*(M)-l_M}$, whence by (\ref{e-tg4}),
$\tg(M)=\beta_1^{r^*(M)}\alpha_1^{i_M}\alpha_2^{r(M)-i_M}$ if $\beta_2=0$. Thus $\tg(M)$ is easy to evaluate if $\alpha_2=0$ or
$\beta_2=0$ (a contrast with the fact that the Tutte polynomial is difficult to evaluate, see \cite{GJ,JVW,V}).

\section{Modifications of the Tutte polynomial}

Let $\xi\in \Bbb Z[x,y]$. Then $\xi$ is a multiple of $1$ whence by Theorem~\ref{mod-t1}, the $\xi$-isthmus-smooth modification of the
Tutte polynomial of $M$ is
\begin{eqnarray}\label{e-def2}
\phantom{aaa}T\isx(M;x,y) = 
\xi^{|E|}(x-1)^{r^*(M)}T(M;x,y)
\end{eqnarray}
and satisfies
\begin{eqnarray}\label{e-cd3}
\begin{array}{lll}
\phantom{aaa}&T\isx(M;x,y) = 1 &\text{if } E=\emptyset ,\\
&T\isx(M;x,y) = \xi y(x-1)T\isx(M-e;x,y) &\text{if $e$ is a loop of $M$,} \\
&T\isx(M;x,y) = \xi T\isx(M/e;x,y) + \xi(x{-}1)T\isx(M{-}e;x,y) \hskip-1mm &\text{otherwise.}
\end{array}
\end{eqnarray}

By Theorem~\ref{mod-t2}, the $\xi$-loop-smooth modification of the Tutte polynomial of $M$ is
\begin{eqnarray}\label{e-def3} \phantom{aaa}
T\lsx(M;x,y) = \xi^{|E|}(y-1)^{r(M)}T(M;x,y)
\end{eqnarray}
and satisfies
\begin{eqnarray}\label{e-cd4}
\begin{array}{lll}
\phantom{aa}&T\lsx (M;x,y) = 1 &\text{if } E=\emptyset ,\\
&T\lsx(M;x,y) = \xi x(y-1) T\lsx(M-e;x,y) &\text{if $e$ is an isthmus,} \\ 
&T\lsx(M;x,y) = \xi (y{-}1)T\lsx(M{-}e;x,y) + \xi T\lsx(M{/}e;x,y) \hskip -1mm &\text{otherwise.}
\end{array}
\end{eqnarray}

By (\ref{e-tg4}), $T\lsx(M;x,y)=(T\lsx)^*(M^*;x,y)$, and by (\ref{e-tg6}), $(T\lsx)^*(M^*;x,y)=(T^*)\isx(M^*;x,y)$. By (\ref{e-def1}) and
(\ref{e-tg1}), we have $T^*(M^*;x,y)=T(M^*;y,x)$, whence $(T^*)\isx(M^*;x,y)=T\isx (M^*;y,x)$, i.e., we have a variant of the covariance
formula
\begin{eqnarray}\label{e-cov6}
T\lsx (M;x,y) =T\isx (M^*;y,x).
\end{eqnarray}

Kook, Reiner, and Stanton \cite{KRS} introduced the convolution formula
\begin{eqnarray*}
T(M;x,y)=\sum_{A\subseteq E}T(M/A;x,0)\cdot T(M|A;0,y),
\end{eqnarray*}
(where $M|A$ and $M/A$ denote the restriction of $M$ to $A$ and the contraction of $A$ from $M$, respectively).
Hence by (\ref{e-def2}) and (\ref{e-def3}),
\begin{eqnarray*}
T(M;x,y)= \sum_{A\subseteq E}\xi^{-|E|}(-1)^{-r(M/A)}T\lsx (M/A;x,0))\cdot (-1)^{-r^*(M|A)}T\isx (M|A;0,y).
\end{eqnarray*}
Since $r^*(M|A)=|A|-r(A)$, $r(M/A)=r(M)-r(A)$, and $2r(A)-r(M)-|A|$ has the same parity as $r(M)+|A|$, we get a variant of the convolution formula
\begin{eqnarray}\label{e-con4}
T(M;x,y)= \xi^{-|E|}(-1)^{r(M)} \sum_{A\subseteq E}(-1)^{|A|}T\lsx (M/A;x,0)\cdot T\isx (M|A;0,y).
\end{eqnarray}

The ring $\Bbb Z[x,y]$ has no divisors of zero, therefore it has a quotient field $\Bbb F[x,y]$, consisting of all rational polynomials
with integral coefficients. Thus, as pointed out in the remark after Theorem~\ref{mod-t1},  for any $\xi\in\Bbb F[x,y]$, we can consider
$\xi$-isthmus- and $\xi$-loop-smooth modifications of the Tutte polynomial, thus also formulas (\ref{e-cov6}), (\ref{e-con4}).

If $\tg$ is a T-G invariant determined by $(\alpha_1,\beta_1,\alpha_2,\beta_2)$ and $\xi,\zeta\in R$, then by Lemma~\ref{tut-p2}, there
exists a T-G invariant determined by $(\xi\alpha_1,\zeta\beta_1,\xi\alpha_2,\zeta\beta_2)$ denoted by $\tg_{\xi,\zeta}$. Clearly,
$\tg_{\xi,\zeta}(M)=\xi^{r(M)}\zeta^{r^*(M)}\tg(M)$ for each matroid $M$. Suppose that If $\tg_{\xi,\zeta}$ is isthmus- and loop-smooth in
the same time. Then $\xi\alpha_1=\zeta\beta_1=\xi\alpha_2+\zeta\beta_2$, whence
$\xi/\zeta=\beta_1/\alpha_1=\beta_2/(\alpha_1-\alpha_2)=(\beta_1-\beta_2)/\alpha_2$, and thus
\begin{eqnarray}\label{e-li}
\beta_1=\alpha_1\beta_2/(\alpha_1-\alpha_2).
\end{eqnarray}
On the other hand (\ref{e-li}) implies $\beta_1/\alpha_1=\beta_2/(\alpha_1-\alpha_2)$ and
$(\beta_1-\beta_2)/\alpha_2=\beta_2/(\alpha_1-\alpha_2)$. Thus (\ref{e-li}) is a necessary and sufficient condition for existence of $\xi$
and $\zeta$ such that $\tg_{\xi,\zeta}$ is an isthmus- and loop-smooth invariant. Therefore this kind of transforation cannot be applied for each
$\tg$. In particular, (\ref{e-li}) is not valid for the Tutte polynomial because $y\neq x/(x-1)$.


\begin{thebibliography}{00}

\bibitem{BEPS} L. Beaudin, J. Ellis-Monaghan, G. Pangborn, and R. Shrock, \textit{A little statistical mechanics for the graph theorist},
Discrete Math. \textbf{310} (2010) 2037--2053.

\bibitem{BO} T. Brylawski and J. Oxley, \textit{The Tutte polynomial and its applications}, in: Matroid Applications, (N. White,
Editor), Cambridge University Press, Cambridge (1992), pp. 123-225.

\bibitem{GJ} L.A. Goldberg and M. Jerrum, \textit{Inapproximability of the Tutte polynomial}, Inform. and Comput. \textbf{206} (2008) 908--929.

\bibitem{JVW} F. Jaeger, D.L. Vertigan, and D.J.A. Welsh, \textit{On the computational complexity of the Jones and Tutte polynomials},
Math. Proc. Cambridge Philos. Soc. \textbf{108} (1990) 35--53.

\bibitem{K} M. Kochol, \textit{Splitting formulas for Tutte-Grothendieck invariants}, manuscript (2014).

\bibitem{KRS} W. Kook, V. Reiner, and D. Stanton, \textit{A convolution formula for the Tutte polynomial}, J. Combin. Theory Ser. B
\textbf{76} (1999) 297--300.

\bibitem{OW} J.G. Oxley and D.J.A. Welsh, \textit{The Tutte polynomial and percolation}, in: Graph Theory and Related Topics, (J.A. Bondy
and U.S.R. Murty, Editors), Academic Press, New York (1979), pp. 329-339.

\bibitem{V} D. Vertigan, \textit{The computational complexity of Tutte invariants for planar graphs}, SIAM J. Comput. \textbf{135} (2005) 690--712.

\bibitem{W} D.~J.~A. Welsh, \textit{Complexity: Knots, Colourings and Counting}, London Math. Soc. Lecture Notes Series
186, Cambridge University Press, Cambridge (1993).

\end{thebibliography}
\end{document}